\documentclass[12pt, a4paper]{article}
\usepackage{ngerman}
\usepackage{amsmath}
\usepackage{amsthm}
\usepackage{amsfonts}
\usepackage{amssymb}
\usepackage{graphicx}
\usepackage{geometry}
\usepackage{a4wide}
\usepackage[latin1]{inputenc}

\DeclareMathOperator{\SL}{SL}

\DeclareMathOperator{\Z}{\mathbb{Z}}
\DeclareMathOperator{\R}{\mathbb{R}}
\DeclareMathOperator{\C}{\mathbb{C}}

\renewcommand{\H}{\mathbb{H}}
\renewcommand{\Re}{\text{Re}}
\DeclareMathOperator{\cusp}{cusp}

\renewcommand{\S}{\mathcal{S}}

	\newtheorem{Satz}{Satz}[section]
	\newtheorem{Theorem}[Satz]{Theorem}
	\newtheorem{Lemma}[Satz]{Lemma}
	\newtheorem{Proposition}[Satz]{Proposition}

	\theoremstyle{definition}

\date{}
\author{Markus Schwagenscheidt}
\title{Nonvanishing and Central Critical Values of Twisted $L$-functions of Cusp Forms on Average}

\begin{document} 
		\maketitle

	
	\begin{abstract}
		Let $f$ be a holomorphic cusp form of integral weight $k \geq 3$ for $\Gamma_{0}(N)$ with nebentypus character $\psi$. Generalising work of Kohnen \cite{KohnenNonvanishing} and Raghuram \cite{Raghuram} we construct a kernel function for the $L$-function $L(f,\chi,s)$ of $f$ twisted by a primitive Dirichlet character $\chi$ and use it to show that the average $\sum_{f \in S_{k}(N,\psi)}\frac{L(f,\chi,s)}{\langle f,f\rangle}\overline{a_{f}(1)}$ over an orthogonal basis of $S_{k}(N,\psi)$ does not vanish on certain line segments inside the critical strip if the weight $k$ or the level $N$ is big enough. As another application of the kernel function we prove an averaged version of Waldspurger's theorem relating the central critical value of the $D$-th twist ($D < 0$ a fundamental discriminant) of the $L$-function of a cusp form $f$ of even weight $2k$ to the square of the $|D|$-th Fourier coefficient of a form of half-integral weight $k+1/2$ associated to $f$ under the Shimura correspondence.
	\end{abstract}
	
	\section{Introduction and Results}
	
	Let $f$ be an elliptic cusp form for $\SL_{2}(\Z)$ and let $L^{*}(f,s)$ denote the completed $L$-function of $f$. It has an analytic continuation to $\C$ and satisfies the functional equation $L^{*}(f,s) = L^{*}(f,k-s)$. The generalised Riemann hypothesis for $L^{*}(f,s)$ states that its only zeros inside the critical strip $(k-1)/2 < \Re(s) < (k+1)/2$ lie on the line $\Re(s) = k/2$. In \cite{KohnenNonvanishing} Kohnen constructed a kernel function $R_{k}(\tau,s)$ for the completed $L$-function $L^{*}(f,s)$ and used it to show that in a certain sense the generalised Riemann hypothesis for $L^{*}(f,s)$ holds on average for large weights. More precisely, he proved:
	
	\begin{Theorem}[\cite{KohnenNonvanishing}]\label{ThmKohnenNonvanishing}
		For every $t_{0} \in \R$ and $\varepsilon > 0$ there exists a constant $C(t_{0},\varepsilon) > 0$ such that for every even integer $k > C(t_{0},\varepsilon)$ the function
		\[
		\sum_{f \in S_{k}(1)}\frac{L^{*}(f,s)}{\langle f,f \rangle}
		\]
		(where the sum runs over an orthogonal basis of normalised Hecke eigenforms for $\SL_{2}(\Z)$) does not vanish at points $s = \sigma+ it_{0}$ with $(k-1)/2 < \sigma < k/2-\varepsilon, k/2 +\varepsilon < \sigma < (k+ 1)/2$.
	\end{Theorem}
	
	The theorem has been generalised by Raghuram \cite{Raghuram} to arbitrary level $N$ and primitive Nebentypus character $\psi$, using essentially the same arguments.
	
	Following the ideas of Kohnen \cite{KohnenNonvanishing}, we consider the kernel function 
	\[
	R_{k,N,\psi}(\tau,s,\chi) = \sum_{m=1}^{\infty}\overline{\chi(m)}m^{s-1}P_{k,N,\psi,m}(\tau)
	\]
	(with $P_{k,N,\psi,m} \in S_{k}(N,\psi)$ being the usual $m$-th Poincar\'e series) for the twisted $L$-function 
	\[
	L(f,\chi,s) = \sum_{m=1}^{\infty}\chi(m)a_{f}(m)m^{-s}
	\]
	of cusp forms $f \in S_{k}(N,\psi)$ of weight $k \geq 3$, level $N$ and Nebentypus character $\psi$, where $\chi$ is a primitive Dirichlet character whose conductor is coprime to the level $N$ and $a_{f}(m)$ denotes the $m$-th Fourier coefficient of $f$. We compute the Fourier expansion of $R_{k,N,\psi}(\tau,s,\chi)$ and use it to show an anlogue of Theorem \ref{ThmKohnenNonvanishing} for $L(f,\chi,s)$. The arguments closely follow those in \cite{KohnenNonvanishing} and \cite{Raghuram}. We prove:
	
	\begin{Theorem}\label{Nonvanishing}
	Let $\chi$ be a primitive Dirichlet character mod $h$.
	\begin{enumerate}
		\item For every positive integer $N$ with $(N,h) =1$, every positive integer $m$ with $(m,h) = 1$, every $t_{0} \in \R$ and $\varepsilon > 0$ there exists a constant $C(t_{0},\varepsilon,N,m) > 0$ such that for every integer $k > C(t_{0},\varepsilon,N,m)$ and every Dirichlet character $\psi$ mod $N$ with $\psi(-1) = (-1)^{k}$ the function
		\[
		\sum_{f \in S_{k}(N,\psi)}\frac{L(f,\chi,s)}{\langle f,f \rangle}\overline{a_{f}(m)}
		\]
		does not vanish at points $s = \sigma+ it_{0}$ with $(k-1)/2 < \sigma < k/2-\varepsilon, k/2 +\varepsilon < \sigma < (k+ 1)/2$.
		\item For every  integer $k \geq 3$, every positive integer $m$ with $(m,h) = 1$, every $t_{0} \in \R$ and $\varepsilon > 0$ there exists a constant $C(t_{0},\varepsilon,k,m) > 0$ such that for every integer $N > C(t_{0},\varepsilon,k,m)$ with $(N,h) = 1$ and every Dirichlet character $\psi$ mod $N$ with $\psi(-1) = (-1)^{k}$ the function
		\[
		\sum_{f \in S_{k}(N,\psi)}\frac{L(f,\chi,s)}{\langle f,f \rangle}\overline{a_{f}(m)}
		\]
		does not vanish at points $s = \sigma+ it_{0}$ with $(k-1)/2 < \sigma < k/2-\varepsilon, k/2 +\varepsilon < \sigma < (k+1)/2$.
		\end{enumerate}
		Here the sums run over an orthogonal basis of $S_{k}(N,\psi)$.
	\end{Theorem}
	
	Note that $L(f,\chi,s)$ and its completion $L^{*}(f,\chi,s)$ differ by a factor which is indepent of $f$ and does not vanish in the critical strip, so we could as well replace $L(f,\chi,s)$ by $L^{*}(f,\chi,s)$ in the theorem.

	As a second application of the kernel function $R_{k,N,\psi}(\tau,s,\chi)$ we prove an averaged version of Waldspurger's theorem relating the central critical value of the twisted $L$-function of a cusp form of even weight to the square of the Fourier coefficient of a modular form of half-integral weight. Let us explain our result in some more detail:
	
	Let $g \in S^{+}_{k+1/2}$ be a cuspidal Hecke eigenform of half-integral weight $k+1/2$ in the plus space of level $4$ and let $f \in S_{2k}$ be the normalized eigenform of even weight $2k$ for $\SL_{2}(\Z)$ associated to $g$ under the Shimura correspondence. Waldspurger's theorem \cite{Waldspurger} as explicated by Kohnen and Zagier \cite{KohnenZagier} states that for a fundamental discriminant $D$ with $(-1)^{k}D > 0$ it holds
	\begin{align}\label{Waldspurger}
	\frac{|c_{g}(|D|)|^{2}}{\langle g,g\rangle} = \frac{(k-1)!}{\pi^{k}}|D|^{k-1/2}\frac{L(f,D,k)}{\langle f,f\rangle}
	\end{align}
	where $L(f,D,s) =\sum_{m = 1}^{\infty}\big(\frac{D}{m} \big)a_{f}(m)m^{-s}$ denotes the $L$-function of $f$ twisted by the Kronecker symbol $\big(\frac{D}{\cdot}\big)$ and $c_{g}(|D|)$ is the $|D|$-th Fourier coefficient of $g$. The theorem has been generalised to odd squarefree level by Kohnen \cite{KohnenFourierCoefficients} and to arbitrary level using Jacobi forms instead of half-integral weight forms by Gross, Kohnen and Zagier \cite{GKZ}.
	 Shimura \cite{Shimura93} gave similar but different formulas for Hilbert modular forms. Recently, Kojima \cite{Kojima} generalised the lifting maps and the Waldspurger formula of $\cite{GKZ}$ to Hilbert modular forms over totally real fields of class number $1$.
	
	We prove an averaged version of Waldspurger's formula for arbitrary level $N$ using the language of Jacobi forms as in \cite{GKZ}, where by 'averaged' we mean the formula obtained by summing both sides of the Jacobi form analogue of (\ref{Waldspurger}) over orthogonal bases of the corresponding spaces of cusp forms.
		\begin{Theorem}\label{AveragedWaldspurger}
		Let $k \geq 2$ and $N$ be positive integers. Let $D < 0$ be a fundamental discriminant with $(D,N) = 1$ and $D \equiv r^{2} (4N)$ for some $r \in \Z$, and let $m > 0$ be a positive integer. Then it holds
		\begin{align}
	\sum_{\phi \in J^{\cusp}_{k+1,N}}\frac{c_{\phi}(D,r)\overline{c_{\phi|T_{m}}(D,r)}}{\langle \phi,\phi \rangle} = \frac{(k-1)!}{2^{2k-1}\pi^{k}N^{k-1}}|D|^{k-1/2}\sum_{f \in S_{2k}(N)}\frac{L(f,D,k)}{\langle f,f\rangle}\overline{a_{f}(m)}
		\end{align}
		where the sums run over orthogonal bases of the spaces of Jacobi cusp forms of weight $k+1$ and index $N$ and elliptic cusp forms of weight $2k$ and level $N$, and $T_{m}$ denotes the $m$-th Hecke operator on $J_{k+1,N}^{\cusp}$.
	\end{Theorem}
	
	We remark that the formula does not mention the Shimura correspondence, that is, the forms $\phi$ and $f$ need not be related to each other. Further, $\phi$ and $f$ are not required to be newforms or Hecke eigenforms. We see that the left hand side does not depend on the choice of $r$ with $r^{2} \equiv D(4N)$, although the coefficient $c_{\phi}(D,r)$ of $\phi$ is in general not independent of $r$.
	
	For the proof of the theorem we look at the Fourier expansion of the kernel function $R_{2k,N,1}\big(\tau,s,\big(\frac{D}{\cdot}\big)\big)$ at the critical point $s = k$ and show that it equals the image of the $(D,r)$-th Poincar\'e series $P^{J}_{k+1,N,(D,r)} \in J^{\cusp}_{k+1,N}$ defined in \cite{GKZ} under the $(D,r)$-th lifting map 
	\begin{align}\label{ShimuraLift}
	\mathcal{S}_{D,r}(\phi) (w) = \sum_{m=1}^{\infty}\left(\sum_{d\mid m}\left(\frac{D}{d}\right)d^{k-1}c_{\phi}\left(\frac{m^{2}}{d^{2}}D, \frac{m}{d}r \right)\right)e(mw)
	\end{align}
	from $J^{\cusp}_{k+1,N}$ to $S_{2k}(N)$ (see \cite{GKZ}).
	Using the Petersson coefficient formulas for the corresponding Poincar\'e series we obtain the equation
	\begin{align*}
	\sum_{\phi \in J^{\cusp}_{k+1,N}}\frac{\overline{c_{\phi}(D,r)}}{\langle \phi,\phi\rangle}\S_{D,r}(\phi) &=\S_{D,r}\big(P^{J}_{k+1,N,(D,r)}\big) \\
	& = R_{2k,N,1}\big(\tau,k,\big(\tfrac{D}{\cdot}\big)\big) = \sum_{f \in S_{2k}(N)}\frac{\overline{L(f,D,k)}}{\langle f,f\rangle}f
	\end{align*}
	and taking the $m$-th Fourier coefficient yields the result.
	
	We would like to emphasise that our proof does not need Hecke theory and does not make use of any properties of the Shimura correspondence.
	
	Iwaniec \cite{Iwaniec} gave another averaged version of Waldspurger's theorem which is very different from our formula, e.g. it involves averaging over forms of different levels.

	Finally we remark that the techniques of this paper can be generalised to yield an averaged Waldspurger formula for Hilbert-Jacobi cusp forms and an analogue of Theorem \ref{Nonvanishing} for Hilbert cusp forms over totally real number fields of narrow class number $1$. The details will be part of the author's PhD thesis.
		
	\section{The kernel function for the twisted $L$-function of a cusp form}
	 
		In the following let $k,N$ be positive integers with $k \geq 3$. Further, let $\psi$ be a Dirichlet character mod $N$ with $\psi(-1) = (-1)^{k}$ and let $\chi$ be a primitive Dirichlet character mod $h$ with $(N,h) =1$. Throughout we write $e(z) := e^{2\pi iz}$ for $z \in \C$.
		
	Inspired by \cite{KohnenNonvanishing} we define the function
	\[
	R_{k,N,\psi}(\tau,s,\chi) = \sum_{m=1}^{\infty}\overline{\chi(m)}m^{s-1}P_{k,N,\psi,m}(\tau)
	\]
	where 
	\[
	P_{k,N,\psi,m}(\tau) = \sum_{\left(\begin{smallmatrix}a & b \\ c & d \end{smallmatrix} \right) \in \Gamma_{\infty}\setminus \Gamma_{0}(N)}\overline{\psi(d)}(c\tau+d)^{-k}e\left(m\frac{a\tau + b}{c\tau + d} \right)
	\]
	with $\Gamma_{\infty} = \left\{\pm\left(  \begin{smallmatrix}1 & n \\ 0 & 1 \end{smallmatrix}\right): n \in \Z\right\}$ is the $m$-th Poincar\'e series in $S_{k}(N,\psi)$. For $1 < \Re(s) < k-1$ the function $R_{k,N,\psi}(\tau,s,\chi)$ defines a cusp form in $S_{k}(N,\psi)$. Using the Petersson coefficient formula
	\begin{align}\label{PeterssonCoefficientFormula}
	\langle f,P_{k,N,\psi,m} \rangle = \frac{\Gamma(k-1)}{(4\pi m)^{k-1}}a_{f}(m)
	\end{align}
	for $f = \sum_{m =1}^{\infty}a_{f}(m)e(m\tau) \in S_{k}(N,\psi)$ we obtain
	\begin{align}\label{KernelEquation1}
	\langle f,R_{k,N,\psi}(\cdot,s,\chi)\rangle = \frac{\Gamma(k-1)}{(4\pi)^{k-1}}L(f,\chi,k-\overline{s}).
	\end{align}
	We now compute the Fourier expansion of $R_{k,N,\psi}(\tau,s,\chi)$:
	
	\begin{Proposition}
		For $1 < \Re(s) < k-1$ the $m$-th Fourier coefficient of $R_{k,N,\psi}(\tau,s,\chi)$ is given by
		\begin{align*}
		&\overline{\chi(m)}m^{s-1} + \delta_{N,1}\chi(-1) i^{-k}h^{2s-k}(2\pi)^{k-2s}\frac{\Gamma(s)}{\Gamma(k-s)} \frac{G(\overline{\chi})}{G(\chi)}\chi(m)m^{k-s-1}\\
		&+\frac{1}{2}i^{-k}(2\pi)^{k-s}\frac{\Gamma(s)}{\Gamma(k)}m^{k-1}\frac{h^{s}}{G(\chi)}\\
		&\times \sum_{\ell (h)}\chi(\ell) \sum_{\substack{a,c \in \Z, (a,c) = 1, N \mid c \\ (ha+\ell c)c > 0 }}\psi(a)c^{-k}\left(\frac{c}{ha+\ell c}\right)^{s}\bigg(e^{\pi i s/2}e^{2\pi im\bar{a}/c} \ _{1}F_{1}\left(s,k;-\frac{2\pi i mh}{(ha+\ell c)c}\right) \\
		&  +\psi(-1)\chi(-1) e^{-\pi i s/2}e^{-2\pi i m\bar{a}/c}\ _{1}F_{1}\left(s,k;\frac{2\pi i mh}{(ha+\ell c)c}\right)\bigg),
		\end{align*}
		where $\delta_{N,1}$ is the Kronecker delta, $\bar{a}$ is any integer with $a\bar{a} \equiv 1 (c)$, $G(\chi) = \sum_{\ell(h)}\chi(\ell)e(\ell/h)$ denotes the Gauss sum of $\chi$ and $_{1}F_{1}(a,b;z)$ is Kummer's confluent hypergeometric function (denoted by $M(a,b,z)$ in \cite{Abramowitz}, 13.1.2).
	\end{Proposition}

	\begin{proof}
		The proof generalises that of Lemma 2 in \cite{KohnenNonvanishing}. We will frequently use the Lipschitz summation formula
		\begin{align}\label{Lipschitz}
		\sum_{m=-\infty}^{\infty}(m+\tau)^{-s} = \frac{(-2\pi i)^{s}}{\Gamma(s)}\sum_{m=1}^{\infty}m^{s-1}e(m\tau)
		\end{align}
		which is valid for $\tau \in \H$ and $\Re(s) > 1$. Here $e(z) = e^{2\pi i z}$ for $z \in \C$. Since $\chi$ is primitive, we have
		\begin{align}\label{CharacterDecomposition}
		\overline{\chi(m)} = \frac{1}{G(\chi)}\sum_{\ell(h)}\chi(\ell)e(\ell m/h).
		\end{align}
		The Lipschitz formula (\ref{Lipschitz}) together with (\ref{CharacterDecomposition}) shows that $R_{k,N,\psi}(\tau,s,\chi)$ can be written as
		\begin{align}
		&\sum_{\left(\begin{smallmatrix}a & b \\ c & d \end{smallmatrix} \right) \in \Gamma_{\infty}\setminus \Gamma_{0}(N)}\overline{\psi(d)}(c\tau + d)^{-k}\sum_{m=1}^{\infty}\overline{\chi(m)}m^{s-1}e\left(m\frac{a\tau+b}{c\tau + d}\right) \\
		&=\frac{\Gamma(s)}{G(\chi)(-2\pi i)^{s}} \sum_{\left(\begin{smallmatrix}a & b \\ c & d \end{smallmatrix} \right) \in \Gamma_{\infty}\setminus \Gamma_{0}(N)}\overline{\psi(d)}(c\tau + d)^{-k}\sum_{\ell(h)}\chi(\ell)\sum_{m=-\infty}^{\infty}\left(\frac{\ell}{h} + m + \frac{a\tau+b}{c\tau +d} \right)^{-s} \\
		&= \frac{1}{2}\frac{\Gamma(s)}{G(\chi)(-2\pi i)^{s}} \sum_{\ell (h)}\chi(\ell)\sum_{\left(\begin{smallmatrix}a & b \\ c & d \end{smallmatrix} \right) \in \Gamma_{0}(N)}\overline{\psi(d)}\tau^{-s}\bigg|_{k}\begin{pmatrix}1 & \frac{\ell}{h} \\ 0 & 1 \end{pmatrix}\begin{pmatrix}a & b \\ c & d \end{pmatrix} \label{Pieces}
		\end{align} 
		where $f|_{k}\left(\begin{smallmatrix}a & b \\ c & d \end{smallmatrix} \right)(\tau) = (c\tau+d)^{-k}f\big(\frac{a\tau+b}{c\tau+d}\big)$ denotes the usual weight $k$ slash operator.
		
		We first compute the Fourier coefficients of the inner sum over $\Gamma_{0}(N)$ for fixed $\ell$ by splitting it into three pieces corresponding to $c = 0, a + \frac{\ell}{h}c = 0$ and $(a + \frac{\ell}{h}c)c \neq 0$.
		
		For $c = 0$ we have to sum over the matrices $\pm\left(\begin{smallmatrix}1 & b \\ 0 & 1 \end{smallmatrix}\right)$ with $b \in \Z$:
		\begin{align*}
		\sum_{\substack{\left( \begin{smallmatrix} a & b \\ c & d \end{smallmatrix}\right) \in \Gamma_{0}(N) \\ c = 0}}\overline{\psi(d)}\tau^{-s}\bigg|_{k}\begin{pmatrix}1 & \frac{\ell}{h} \\ 0 & 1 \end{pmatrix}\begin{pmatrix} a & b \\ c & d \end{pmatrix} &= 2\sum_{b = -\infty}^{\infty}\left(b+\tau+ \frac{\ell}{h} \right)^{-s}  \\
		&= 2\frac{(-2\pi i )^{s}}{\Gamma(s)}\sum_{m=1}^{\infty}m^{s-1}e(m\tau)e\left(\frac{m\ell}{h}\right)
		\end{align*}
		where we again used the Lipschitz formula and the fact that two matrices with opposite sign give the same contribution. 
		
		Now suppose we have a matrix $\left(\begin{smallmatrix}a & b \\ c & d \end{smallmatrix}\right) \in \Gamma_{0}(N)$ with $a + \frac{\ell}{h}c = 0$. We may assume that $\ell$ and $h$ are coprime, since otherwise $\chi(\ell) = 0$. Then the only integral solutions of $a + \frac{\ell}{h}c = 0$ with $(a,c) = 1$ are $a = \pm \ell, c = \mp h$. This is only possible if $N \mid h$, and as $(N,h) = 1$, it is only possible for $N = 1$. In this case we write
		\begin{align*}
		\begin{pmatrix}1 & \frac{\ell}{h} \\ 0 & 1\end{pmatrix} &= 
		\begin{pmatrix}0 & \frac{1}{h} \\ -h & 0 \end{pmatrix}
		\begin{pmatrix}1 & -\frac{\bar{\ell}}{h} \\ 0 & 1 \end{pmatrix}
		\begin{pmatrix} \bar{\ell} & \frac{\ell\bar{\ell} - 1}{h} \\ h & \ell \end{pmatrix}
		\end{align*}
		where $\bar{\ell}$ is an integer with $\ell \bar{\ell} = 1 (h)$. Note that $\left(\begin{smallmatrix} \bar{\ell} & \frac{\ell\bar{\ell} - 1}{h} \\ h & \ell \end{smallmatrix} \right)\in \Gamma_{0}(N) = \SL_{2}(\Z)$. So if $N = 1$ the character $\psi$ is trivial and we obtain
		\begin{align*}
		&\sum_{\substack{\left(\begin{smallmatrix}a & b \\ c & d \end{smallmatrix}\right) \in \Gamma_{0}(N) \\ a + \frac{\ell}{h}c = 0}}\tau^{-s}\bigg|_{k}\begin{pmatrix}1 & \frac{\ell}{h} \\ 0 & 1 \end{pmatrix}\begin{pmatrix} a & b \\ c & d \end{pmatrix}  \\
		&= (-1)^{-s}h^{2s-k}\sum_{\substack{\left(\begin{smallmatrix}a & b \\ c & d \end{smallmatrix}\right) \in \Gamma_{0}(N) \\ a + \frac{\ell}{h}c = 0}}\tau^{s-k}\bigg|_{k} \begin{pmatrix}1 & -\frac{\bar{\ell}}{h} \\ 0 & 1 \end{pmatrix}\begin{pmatrix}\bar{\ell} & \frac{\ell\bar{\ell}-1}{h} \\ h & \ell\end{pmatrix}\begin{pmatrix} a & b \\ c & d \end{pmatrix} \\
		&=  (-1)^{-s}h^{2s-k}\sum_{\substack{\left(\begin{smallmatrix}a & b \\ c & d \end{smallmatrix}\right) \in \Gamma_{0}(N) \\ c = 0}}\tau^{s-k}\bigg|_{k} \begin{pmatrix}1 & -\frac{\bar{\ell}}{h} \\ 0 & 1 \end{pmatrix}\begin{pmatrix} a & b \\ c & d \end{pmatrix}.
		\end{align*}
		Now the same calculation as in the case $c = 0$ but with $s$ replaced by $k-s$ shows that for $\Re(s) < k-1$ the last expression equals
		\[
		2(-1)^{-s}h^{2s-k}\frac{(-2\pi i)^{k-s}}{\Gamma(k-s)}\sum_{m=1}^{\infty}m^{k-s-1}e(m\tau)e\left( -\frac{m\bar{\ell}}{h}\right).
		\]

		For $(a + \frac{\ell}{h}c)c \neq 0$ we have to compute the Fourier integral
		\begin{align}\label{FourierIntegral}
		\int_{iC}^{iC+1}\sum_{\substack{\left(\begin{smallmatrix}a & b \\ c & d \end{smallmatrix} \right) \in \Gamma_{0}(N) \\ (a + \frac{\ell}{h}c)c \neq 0}}\overline{\psi(d)}(c\tau +d)^{-k}\left(\frac{(a + \frac{\ell}{h}c)\tau + b + \frac{\ell}{h}d}{c\tau + d}\right)^{-s}e(-m\tau)d\tau, \qquad C > 0.
		\end{align}
		The calculation can be done in the same way as in the case $ac \neq 0$ in the proof of Lemma 2 in \cite{KohnenNonvanishing}, so we omit the details. The result for the part with $(a + \frac{\ell}{h}c)c > 0$ is
		\[
		i^{-k}\frac{(2\pi)^{k}}{\Gamma(k)}m^{k-1}\sum_{\substack{(a,c) =1 , N\mid c \\ (a + \frac{\ell}{h}c)c > 0}}\psi(a)c^{-k}e(m\bar{a}/c)\left(\frac{c}{(a + \frac{\ell}{h}c)c} \right)^{s} \ _{1}F_{1}\left(s,k;-\frac{2\pi i m}{(a + \tfrac{\ell}{h}c)c}\right)
		\]
		where $a\bar{a} = 1 (c)$. If we replace $a$ by $-a$ and $b$ by $-b$ in the part with $(a + \frac{\ell}{h}c)c < 0$ in (\ref{FourierIntegral}), the sum runs over all integral matrices $\left(\begin{smallmatrix} a & b \\ c & d \end{smallmatrix}\right)$ with determinant $ad - bc= -1$, $N \mid c$ and $(a - \frac{\ell}{h}c)c > 0$. Now this part can be computed in the same way as the part for $(a+\frac{\ell}{h}c)c > 0$ and gives the contribution
		\[
		i^{-k}(-1)^{-s}\frac{(2\pi)^{k}}{\Gamma(k)}m^{k-1}\sum_{\substack{(a,c) =1 , N\mid c \\ (a - \frac{\ell}{h}c)c > 0}}\psi(-a)c^{-k}e(-m\bar{a}/c)\left(\frac{c}{(a - \frac{\ell}{h}c)c} \right)^{s} \ _{1}F_{1}\left(s,k;\frac{2\pi i m}{(a - \tfrac{\ell}{h}c)c}\right).
		\]

		We now insert everything into (\ref{Pieces}). In the pieces corresponding to $(a + \frac{\ell}{h}c)c = 0$ and $(a + \frac{\ell}{h}c)c < 0$ we sum over $-\bar{\ell}$ and $-\ell$ instead of $\ell$, respectively, giving a factor $\overline{\chi(-1)} = \chi(-1)$ in both cases. Taking into account (\ref{CharacterDecomposition}) and $(-i)^{-s} = e^{\pi i  s/2}, (-1)^{-s} = e^{- \pi i s},$ we obtain the stated Fourier expansion.
	\end{proof}
		
	\begin{proof}[Proof of Theorem \ref{Nonvanishing}] We follow the arguments of \cite{KohnenNonvanishing}: Let $\chi$ be a primitive Dirichlet character mod $h$ as above. From (\ref{KernelEquation1}) it follows that
	\begin{align}\label{KernelEquation2}
	\sum_{f \in S_{k}(N,\psi)}\frac{\overline{L(f,\chi,k-\overline{s})}}{\langle f,f\rangle}f(\tau) = \frac{(4\pi)^{k-1}}{\Gamma(k-1)}R_{k,N,\psi}(\tau,s,\chi),
	\end{align}
	where the sum is taken over an orthogonal basis of $S_{k}(N,\psi)$. Suppose that $s_{0} = k/2 - \delta - it_{0}$ with $\varepsilon < \delta \leq \frac{1}{2}$ is a zero of the $m$-th Fourier coefficient (with $(m,h) = 1$) of the left-hand side of (\ref{KernelEquation2}). Then also the $m$-th coefficient of $R_{k,N,\psi}(\tau,s_{0},\chi)$ vanishes:
	\begin{align*}
		&-\overline{\chi}(m)m^{k/2-\delta-it_{0}-1} \\
		&=\delta_{N,1}\chi(-1) i^{-k}(2\pi/h)^{2\delta+2it_{0}}\frac{\Gamma(k/2-\delta-it_{0})}{\Gamma(k/2+\delta+it_{0})} \frac{G(\overline{\chi})}{G(\chi)}\chi(m)m^{k/2+\delta+it_{0}-1}\\
		&+\frac{1}{2}i^{-k}(2\pi)^{k/2+\delta+it_{0}}\frac{\Gamma(k/2-\delta-it_{0})}{\Gamma(k)}m^{k-1}\frac{h^{k/2-\delta-it_{0}}}{G(\chi)}\\
		&\times \sum_{\ell (h)}\chi(\ell) \sum_{\substack{a,c \in \Z, (a,c) = 1, N \mid c \\ (ha+\ell c)c > 0 }}\psi(a)c^{-k/2-\delta-it_{0}}(ha+\ell c)^{-k/2+\delta+it_{0}}\\
		&\bigg(e^{\pi i (k/2-\delta-it_{0})/2}e^{2\pi i\bar{a}/c} \ _{1}F_{1}\left(k/2-\delta-it_{0},k;-\frac{2\pi imh}{(ha+\ell c)c}\right)   \\
		&+\psi(-1)\chi(-1) e^{-\pi i (k/2-\delta-it_{0})/2}e^{-2\pi i \bar{a}/c}\ _{1}F_{1}\left(k/2-\delta-it_{0},k;\frac{2\pi im h}{(ha+\ell c)c}\right)\bigg).
		\end{align*}
		
	For $\Re(\beta) > \Re(\alpha) > 0$ the integral representation
	\[
	_{1}F_{1}(\alpha,\beta;z) = \frac{\Gamma(\beta)}{\Gamma(\alpha)\Gamma(\beta-\alpha)}\int_{0}^{1}e^{tz}t^{\alpha-1}(1-t)^{\beta-\alpha-1}dt
	\]
	from \cite{Abramowitz}, 13.2.1, gives the estimate
	\begin{align}\label{HypergeometricFunctionEstimate}
	|_{1}F_{1}(\alpha,\beta;2\pi i x)| \leq \frac{\Gamma(\beta)}{\Gamma(\alpha)\Gamma(\beta-\alpha)}
	\end{align}
	for $\Re(\alpha) > 1, \Re(\beta-\alpha) > 1$ and $x \in \R$.
		
	We now take absolute values and divide everything by $m^{k/2-1}$. Using $|G(\chi)| = |G(\overline{\chi})| = \sqrt{h}$ and the estimate above we obtain
	\begin{align}\label{Estimate}
	m^{-\delta} &\leq \delta_{N,1}(2\pi/h)^{2\delta}m^{\delta}\frac{|\Gamma(k/2-\delta-it_{0})|}{|\Gamma(k/2+\delta+it_{0})|} + N^{-k/2-\delta}\frac{(2\pi)^{k/2}m^{k/2}h^{k/2}}{|\Gamma(k/2+\delta+it_{0})|}K
	\end{align}
	with some constant $K > 0$ independent of $k$ and $N$. We have
	\[
	\frac{\Gamma(k/2-\delta-it_{0})}{\Gamma(k/2+\delta+it_{0})}(k/2)^{2\delta+2it_{0}} \to 1 \quad (k \to \infty)
	\]
	uniformly in $\delta$ by \cite{Abramowitz}, 6.1.47, hence the first summand in (\ref{Estimate}) tends to $0$ as $k \to \infty$. For fixed level $N$ the second summand in (\ref{Estimate}) goes to $0$ for $k \to \infty$ as well, hence we get a contradiction for large $k$.
	
	If we fix the weight $k$, the first summand in $(\ref{Estimate})$ is $0$ for $N \geq 2$, and the second summand tends to $0$ for $N \to \infty$, also giving a contradiction.
	
	For $s_{0} = k/2 + \delta - it_{0}$ the statement follows for $N = 1$ by the functional equation of (the completion of) $L(f,\chi,s)$ and for $N > 1$ by the same arguments as above.
	\end{proof}
	
	\section{Waldspurger's Theorem on Average}
	
	Now we take $\psi = 1$, $h = |D|$ where $D$ is a negative fundamental discriminant with $(D,N) = 1$, and $\chi = \big( \frac{D}{\cdot}\big)$. Note that for a negative fundamental discriminant $D$ the Kronecker symbol $\big(\frac{D}{\cdot} \big)$ is a primitive quadratic Dirichlet character mod $|D|$ with $\big( \frac{D}{-1}\big) = - 1$ and Gauss sum $G\big(\big(\frac{D}{\cdot}\big)\big):= \sum_{\ell(D)}\big(\frac{D}{\ell}\big)e(\ell/|D|) = i\sqrt{|D|}$. We will show that
	\begin{align}\label{PoincareIdentity}
		R_{2k,N}\big(\tau,k,\big(\tfrac{D}{\cdot}\big)\big) = \mathcal{S}_{D,r}\left(P^{J}_{k+1,N,(D,r)}\right)
	\end{align}
	where we dropped $\psi = 1$ from the notation and where $r \in \Z$ is chosen such that $r^{2} \equiv D (4N)$. Here $\mathcal{S}_{D,r}$ is the lifting map (\ref{ShimuraLift}) from $J_{k+1,N}^{\cusp}$ to $S_{2k}(N)$ and $P^{J}_{k+1,N,(D,r)} \in J^{\cusp}_{k+1,N,(D,r)}$ is the $(D,r)$-th Poincar\'e series for the Jacboi group characterised by the coefficient formula
	\begin{align}\label{JacobiPoincareSeries}
	\langle \phi,P^{J}_{k+1,N,(D,r)} \rangle = \frac{N^{k-1}\Gamma(k-1/2)}{2\pi^{k-1/2}}|D|^{-k+1/2}c_{\phi}(D,r)
	\end{align}
	for every cusp form $\phi = \sum_{D < 0}\sum_{\substack{r^{2} \equiv D (4N)}}c_{\phi}(D,r)e\left(\frac{r^{2}-D}{4N}\tau + rz \right) \in J_{k+1,N}^{\cusp}$ (see \cite{GKZ}, Proposition II.1).

	\begin{Proposition}
		The $m$-th Fourier coefficient of $R_{2k,N}\big(\tau,k,\big(\frac{D}{\cdot}\big)\big)$ is given by
		\begin{align*} 
		(1\pm \delta_{N,1})\left( \frac{D}{m}\right)m^{k-1} + i^{k+1}\pi\sqrt{2}m^{k-1/2}\sum_{\substack{n\geq 1 \\ N \mid n}}n^{-1/2}K^{\pm}_{N,n}(m,D)J_{k-1/2}\left( \frac{\pi}{n}m|D|\right)
		\end{align*}
		where $\pm 1 = (-1)^{k+1}$, $\delta_{N,1}$ is the Kronecker delta, $J_{k-1/2}$ is the Bessel function of order $k-1/2$,
		\[
		K_{N,n}(m,D) = \sum_{\ell(D)}\left(\frac{D}{\ell} \right)\sum_{\substack{(a,c) \in \Z^{2} \\ (a,c) = 1,  c > 0, N \mid c\\ (|D|a + \ell c)c = n}}e_{2n}\left(m\left(|D| - 2(|D|a + \ell c)\bar{a}\right) \right)
		\]
		with $a\bar{a} \equiv 1 (c)$ is a finite exponential sum, $K^{\pm}_{N,n}(m,D) = K_{N,n}(m,D) \pm K_{N,n}(-m,D)$ and $e_{2n}(z) = e^{2\pi i z/2n}$ for $z \in \C$.
	\end{Proposition}
	
	\begin{proof}
		The hypergeometric function appearing in the Fourier expansion of $R_{2k,N}\big(\tau,k,\big(\frac{D}{\cdot}\big)\big)$ is related to the Bessel function of order $k-1/2$:
		\begin{align*}
	&_{1}F_{1}\left(k,2k;\frac{2\pi i m |D|}{(|D|a+\ell c)c}\right) \\
	&\qquad= \Gamma(k+1/2)e\left(\frac{m|D|}{2(|D|a + \ell c)c}\right)\left( \frac{1}{2}\frac{\pi m |D|}{(|D|a + \ell c)c}\right)^{-k+1/2}J_{k-1/2}\left( \frac{\pi m |D|}{(|D|a + \ell c)c}\right),
	\end{align*}
	see \cite{Abramowitz}, 13.6.1. Using the Kummer transformation $_{1}F_{1}(a,b;z) = e^{z}\ _{1}F_{1}(b-a,b;-z)$, see \cite{Abramowitz}, 13.1.27, we get
	\begin{align*}
	&_{1}F_{1}\left(k,2k;-\frac{2\pi i m |D|}{(|D|a+\ell c)c}\right) = e\left(-\frac{m |D|}{(|D|a + \ell c)c}\right) \\
	&\qquad\times \Gamma(k+1/2)e\left(\frac{m|D|}{2(|D|a + \ell c)c}\right)\left( \frac{1}{2}\frac{\pi m |D|}{(|D|a + \ell c)c}\right)^{-k+1/2}J_{k-1/2}\left( \frac{\pi m |D|}{(|D|a + \ell c)c}\right).
	\end{align*}
	The formula now follows by a straightforward calculation using the Legendre duplication formula $\Gamma(k)\Gamma(k+1/2) = 2^{1-2k}\sqrt{\pi}\Gamma(2k)$.
	\end{proof}
	
	In order to prove the identity (\ref{PoincareIdentity}) we compute the Fourier expansion of the right-hand side. To this end, we need the Fourier expansion of the Poincar\'e series $P^{J}_{k+1,N,(D,r)} \in J^{\cusp}_{k+1,N}$.
	
	\begin{Proposition}[\cite{GKZ}, Proposition II.2]
		The Poincar\'e series $P^{J}_{k+1,N,(D,r)} \in J^{\cusp}_{k+1,N}$ has the expansion
		\[
		P^{J}_{k+1,N,(D,r)}(\tau,z) = \sum_{\substack{D' < 0, r' \in \Z \\ r'^{2} \equiv D'(4N)}}g^{\pm}_{k+1,N,(D,r)}(D',r')e\left(\frac{r'^{2}-D'}{4N}\tau + r'z \right) 
		\]
		where $\pm 1 = (-1)^{k+1}, g^{\pm}_{k+1,N,(D,r)}(D',r') = g_{k+1,N,(D,r)}(D',r') \pm g_{k+1,N,(D,r)}(D',-r')$ and
		\begin{align*}
		g_{k+1,N,(D,r)}(D',r') &= \delta_{N}(D,r,D',r') + i^{k+1}\pi\sqrt{2}N^{-1/2}(D'/D)^{\frac{k}{2}-\frac{1}{4}} \\
			& \qquad \times \sum_{n\geq 1}H_{N,n}(D,r,D',r')J_{k-1/2}\left(\frac{\pi}{Nn}\sqrt{D'D}\right)
		\end{align*}
		where
		\[
		\delta_{N}(D,r,D',r') = \begin{cases}
		1, & \text{if } D' = D \text{ and } r' \equiv r (2N), \\
		0, & \text{otherwise },
		\end{cases}
		\]
		and
		\[
		H_{N,n}(D,r,D',r') = n^{-3/2}\sum_{\substack{\rho(n)^{*} \\ \lambda(n)}}e_{n}\left(\left(N\lambda^{2} + r\lambda + \frac{r^{2}-D}{4N}\right)\overline{\rho} + \frac{r'^{2}-D'}{4N}\rho + r'\lambda\right)e_{2Nn}(rr')
		\]
		with $\rho\bar{\rho} \equiv 1 (n)$ is a Kloosterman-type sum.
	\end{Proposition}
		
	\begin{Proposition}
		The $m$-th Fourier coefficient of $\S_{D,r}\left(P^{J}_{k+1,N,(D,r)}\right)$ is given by
		\begin{align*} 
		&(1 \pm \delta_{N,1})\left( \frac{D}{m}\right)m^{k-1} + i^{k+1}\pi\sqrt{2}m^{k-1/2}\sum_{\substack{n\geq 1 \\ N\mid n}}n^{-1/2}S^{\pm}_{N,n}(m,D)J_{k-1/2}\left( \frac{\pi}{n}m|D|\right)
		\end{align*}
		where $\pm 1 = (-1)^{k+1}$, $\delta_{N,1}$ is the Kronecker delta, $J_{k-1/2}$ is the Bessel function of order $k-1/2$,
		\[
		S_{N,n}(m,D) = \sum_{\substack{b(2n) \\  b^{2}\equiv D^{2}(4n) \\ b \equiv D(2N)}}\chi_{D}\left(\begin{pmatrix} n & b/2 \\ b/2 & \frac{b^{2}-D^{2}}{4n}\end{pmatrix} \right)e_{2n}(bm)
		\]
		with the genus character $\chi_{D}$ defined in \cite{GKZ}, $S^{\pm}_{N,n}(m,D) = S_{N,n}(m,D) \pm S_{N,n}(-m,D)$ and $e_{2n}(z) = e^{2\pi i z/2n}$ for $z \in \C$.
	\end{Proposition}
	
	\begin{proof}
		The $m$-th coefficient of $\S_{D,r}\left(P^{J}_{k+1,N,(D,r)}\right)$ is
		\[
		\sum_{d\mid m}\left( \frac{D}{d}\right)d^{k-1}g^{\pm}_{k+1,N,(D,r)}\left(\frac{m^{2}}{d^{2}}D,\frac{m}{d}r \right)
		\]
		where $g^{\pm}_{k+1,N,(D,r)}(D',r')$ denotes the $(D',r')$-th coefficient of $P^{J}_{k+1,N,(D,r)}$.

		The $\delta_{N}$-part only gives a contribution for $d = m$ in which case $\delta_{N}(D,r,D,r) = 1$. Further, $\delta_{N}(D,r,D,-r)$ is $1$ if and only if $N \mid r$ which is equivalent to $N \mid D$ since $D$ is a fundamental discriminant and $r^{2} \equiv D(4N)$. Since $(D,N) = 1$ by assumption, this is only possible for $N = 1$. 

		The Kloosterman sum part of $g_{k+1,N,(D,r)}$ gives the contribution
		\begin{align*}
		i^{k+1}\pi \sqrt{2}N^{-1/2}m^{k-1/2}\sum_{n \geq 1}\sum_{d\mid (m,n)}\left(\frac{D}{d}\right)d^{-1/2}H_{N,n/d}\left(D,r,\frac{m^{2}}{d^{2}}D,\frac{m}{d}r\right)J_{k-1/2}\left( \frac{\pi}{Nn}m|D| \right).
		\end{align*}
		The proof will be finished by the following lemma from \cite{GKZ}.
	\end{proof}
	
	\begin{Lemma}[\cite{GKZ}, Lemma II.3]
		For all $m \geq 1, n\geq 0, r \in \Z$ with $D = r^{2}-4Nn < 0$ we have
		\[
		S_{N,Nn}(m,D) = \sum_{d\mid (m,n)}\left(\frac{D}{d}\right)(n/d)^{1/2}H_{N,n/d}\left(D,r,\frac{m^{2}}{d^{2}}D,\frac{m}{d}r\right).
		\]
	\end{Lemma}

	Comparing the Fourier coefficients of $R_{2k,N}\big(\tau,k,\big(\frac{D}{\cdot}\big)\big)$ and $\S_{D,r}(P^{J}_{k+1,N,(D,r)})$ we have computed above we see that we have to show
	\[
	S_{N,n}(m,D) = K_{N,n}(m,D)
	\]
	for all $n \geq 1$ with $N \mid n, (D,N) = 1$, and all $m \in \Z$. This immediatly follows from the next lemma:
	
	\begin{Lemma}
		Let $N \mid n$ and $(D,N) = 1$. A set of representatives for $b(2n)$ with $b^{2} \equiv D^{2} (4n)$ and $b \equiv D (2N)$ is given by
		\[
		b = |D| - 2(|D|a + \ell c)\bar{a}
		\]
		where $a,c$ run through $\Z$ with $(a,c) = 1$, $c > 0$ and $N \mid c$, and $\ell$ runs mod $D$ such that $(|D|a + \ell c)c = n$. Here $\bar{a}$ is any fixed integer with $a \bar{a} \equiv 1 (c)$. If $b$ is of this form, we have
		\[
		\chi_{D}\left( \begin{pmatrix}n & b/2 \\ b/2 & \frac{b^{2}-D^{2}}{4n}\end{pmatrix}\right) = \left( \frac{D}{\ell}\right).
		\]
 	\end{Lemma}
	
	\begin{proof}
	Recall that for every class $\rho$ mod $2N$ the group $\Gamma_{0}(N)$ acts on the set of quadratic forms
	\[
	\mathcal{Q}_{N,D,\rho} = \left\{ Q = \begin{pmatrix}a & b/2 \\ b/2 & c \end{pmatrix}: N \mid a, \ b^{2}-4ac = D^{2}, \ b \equiv \rho (2N) \right\} 
	\] 
	by $Q\circ M =  M^{t}QM$. The elements $b (2n)$ with $b^{2} \equiv D^{2}(4n)$ and $b \equiv D(2N)$ can be thought of as the upper right entries of a full system of $\Gamma_{0}(N)$-inequivalent quadratic forms in $\mathcal{Q}_{N,D^{2},D}$ with fixed upper left entry $n$. A system of $\Gamma_{0}(N)$-representatives for the whole set $\mathcal{Q}_{N,D^{2},D}$ is given by $\left(\begin{smallmatrix}0 & D/2 \\ D/2 & \ell \end{smallmatrix}\right)$ where $\ell$ runs mod $D$ (here the condition $(D,N) = 1$ is used). Hence we need to find all $\left(\begin{smallmatrix}\alpha & \beta \\ \gamma & \delta \end{smallmatrix}\right) \in \Gamma_{0}(N)$ such that the matrices
	\begin{align}\label{Matrix}
	\begin{pmatrix}\alpha & \gamma \\ \beta & \delta \end{pmatrix}
	\begin{pmatrix}0 & D/2 \\ D/2 & \ell \end{pmatrix}
	\begin{pmatrix}\alpha & \beta \\ \gamma & \delta \end{pmatrix} = \begin{pmatrix}(\alpha D + \gamma \ell)\gamma & \frac{-D+2(\alpha D + \gamma \ell)\delta}{2} \\ \frac{-D+2(\alpha D + \gamma \ell)\delta}{2} & (\beta D + \delta \ell)\delta \end{pmatrix}
	\end{align}
	have upper left entry $n$ and are inequivalent mod $\Gamma_{0}(N)$. Setting $a = -\alpha, c = \gamma, \bar{a}=-\delta$ and changing $a$ and $c$ to $-a$ and $-c$ if $c < 0$ it is now easy to see that we can take $b = |D|-2(|D|a + \ell c)\bar{a}$ with $a,c,\ell$ as stated in the lemma. 
	
	The genus character $\chi_{D}$ of the matrix (\ref{Matrix}) and the Kronecker symbol $\left( \frac{D}{\ell}\right)$ both are $0$ for $(D,\ell) > 1$. For $(D,\ell) = 1$ the quadratic form (\ref{Matrix}) properly represents $\ell$, so $\chi_{D}$ applied to the matrix (\ref{Matrix}) equals $\left(\frac{D}{\ell} \right)$ by definition of the genus character.
	\end{proof}
	
	\begin{proof}[Proof of Theorem \ref{AveragedWaldspurger}]
%
		As in (\ref{KernelEquation2}) we write
		\begin{align*}
	R_{2k,N}\big(\tau,s,\big(\tfrac{D}{\cdot}\big)\big) =\frac{\Gamma(2k-1)}{(4\pi)^{2k-1}}\sum_{f \in S_{2k}(N)}\frac{\overline{L(f,D,2k-\overline{s})}}{\langle f,f\rangle}f(\tau),
	\end{align*}
		with $f$ running through an orthogonal basis of $S_{2k}(N)$. On the other hand, the Petersson coefficient formula (\ref{JacobiPoincareSeries}) gives
		\begin{align*}
		\S_{D,r}\left(P^{J}_{k+1,N,(D,r)}\right) &=  \frac{N^{k-1}\Gamma(k-1/2)}{2\pi^{k-1/2}|D|^{k-1/2}}\sum_{\phi \in J^{\cusp}_{k+1,N}}\frac{\overline{c_{\phi}(D,r)}}{\langle \phi,\phi \rangle}\mathcal{S}_{D,r}(\phi)(\tau)
		\end{align*}
		where $\phi$ runs through an orthogonal basis of $J^{\cusp}_{k+1,N}$ and $c_{\phi}(D,r)$ is the $(D,r)$-th coefficient of $\phi$. As we have shown above, both expressions are equal at $s = k$. By \cite{EichlerZagier}, Theorem 4.5, the $m$-th Fourier coefficient of $\S_{D,r}(\phi)$ is just the $(D,r)$-th coefficient of $\phi|T_{m}$, so taking the $m$-th coefficient (and complex conjugation) yields
		\[
		\frac{\Gamma(2k-1)}{(4\pi)^{2k-1}}\sum_{f \in S_{2k}(N)}\frac{L(f,D,s)}{\langle f,f\rangle}\overline{a_{f}(m)} =  \frac{N^{k-1}\Gamma(k-1/2)}{2\pi^{k-1/2}|D|^{k-1/2}}\sum_{\phi \in J^{\cusp}_{k+1,N}}\frac{c_{\phi}(D,r)}{\langle \phi,\phi \rangle}\overline{c_{\phi|T_{m}}(D,r)}
		\]
		Now the Legendre duplication formula $\Gamma(k-1/2)\Gamma(k) = 2^{2-2k}\sqrt{\pi}\Gamma(2k-1)$ completes the proof of Theorem \ref{AveragedWaldspurger}.
		\end{proof}

\bibliography{references}{}
\bibliographystyle{alphadin}

\end{document}